\documentclass[11pt,twoside,a4paper]{article}
\usepackage{a4wide}

\usepackage{amssymb,amsmath,amsthm}

\usepackage{graphicx}

\usepackage[colorlinks=true,citecolor=blue,urlcolor=blue,linkcolor=blue,bookmarksopen=true,unicode=true,pdffitwindow=true]{hyperref}

\newcommand{\set}[1]{\ensuremath{\left\{#1 \right\}}}
\newcommand{\setcond}[2]{\set{#1 \, ; \: #2}}

\newcommand{\cG}{\mathcal{G}}
\newcommand{\GF}{\mathcal{G_F}}
\newcommand{\cS}{\mathcal{S}}
\newcommand{\SF}{\mathcal{S_F}}
\newcommand{\cN}{\mathcal{N}}
\newcommand{\NF}{\mathcal{N_F}}
\newcommand{\cH}{\mathcal{H}}
\newcommand{\HF}{\mathcal{H_F}}

\newtheorem{theorem}{Theorem}[section]

\newtheorem{lemma}[theorem]{Lemma}

\title{The realization graph of every degree sequence\\ has a Hamilton path}
\author{Petr Hladík, Ji\v{r}\'i Fink \footnote{E-mail: \href{mailto:hladik.petr@outlook.com}{hladik.petr@outlook.com}, \href{mailto:fink@ktiml.mff.cuni.cz}{fink@ktiml.mff.cuni.cz}} \\ \small Department of Theoretical Computer Science and Mathematical Logic\\\small Faculty of Mathematics and Physics, Charles University in Prague}
\date{}

\begin{document}
\maketitle

\begin{abstract}
For a degree sequence $d$, the \emph{realization graph} $\GF(d)$ is the graph whose vertices are the labeled realizations of $d$, two of which are adjacent if they differ by a single $2$-switch.
We prove that for every degree sequence $d$ and every realization $S$ of $d$, the graph $\GF(d)$ contains a Hamilton path starting at $S$.
This answers a question of Barrus (2016), which was also raised independently by M\"utze (2023) in his survey of combinatorial Gray codes.
As a consequence, an embedding observation of Arikati and Peled (1999) implies that for any vectors $R$ and $C$ of non-negative integers, the interchange graph $\mathcal{A}_{\mathcal F}(R,C)$ of $(0,1)$-matrices with row sums $R$ and column sums $C$ contains a Hamilton path starting at any prescribed matrix, thereby resolving a question of Brualdi (1980).
\end{abstract}

\section{Introduction}

A sequence $d = (d_1, \ldots, d_n)$ of non-negative integers is called a \emph{degree sequence} if there exists a simple labeled graph on the vertex set $[n] = \{1,\ldots,n\}$ in which every vertex $i \in [n]$ has degree $d_i$; such a graph is called a \emph{realization of $d$}.
Let $\cG(d)$ denote the family of edge sets $E \subseteq \binom{[n]}{2}$ such that the graph $([n], E)$ is a realization of $d$. We work with edge sets rather than with the graphs themselves, purely for notational convenience.
The existence of a realization was first characterized by Havel~\cite{H55} and, later, by Erd\H{o}s and Gallai~\cite{EG60}; Havel~\cite{H55} and Hakimi~\cite{H62} also gave constructive algorithms, which were subsequently extended to directed graphs by Erd\H{o}s, Mikl\'os, and Toroczkai~\cite{EMT10}.
The degree sequences with a unique realization, known as \emph{threshold graphs}, were characterized by Chv\'atal and Hammer~\cite{CH77}; see the book by Mahadev and Peled~\cite{MP95}.

Two distinct realizations of the same degree sequence differ in at least four edges. The minimal case corresponds to replacing two edges $uv$ and $wx$ by the edges $ux$ and $vw$.
This operation, called a $2$-switch on the vertices $u,v,w,x$ (applied only when all four vertices are distinct and the resulting graph is simple), preserves the degree of every vertex and hence maps realizations to realizations.
The \emph{realization graph} $\GF(d)$ (also called the flip graph in the Gray-code literature) has vertex set $\cG(d)$, with an edge between two realizations if they differ by a single $2$-switch; see Figure~\ref{figure:gf}.

Our main result is the following.

\begin{theorem}\label{thm_final}
For every degree sequence $d$ and every realization $S \in \cG(d)$, the realization graph $\GF(d)$ contains a Hamilton path starting at $S$.
\end{theorem}

Since a Hamilton path starting at an arbitrary prescribed realization is stronger than the mere existence of a Hamilton path, Theorem~\ref{thm_final} settles the path version of a question that was asked for all degree sequences by Barrus~\cite{BA16} and, independently, by M\"utze~\cite{mutze2022combinatorial} in his survey of combinatorial Gray codes.

The structure of $\GF(d)$ has attracted considerable attention.
Fulkerson, Hoffman, and McAndrew~\cite{FHM65} proved that $\GF(d)$ is connected for every degree sequence $d$.
Erd\H{o}s, Kir\'aly, and Mikl\'os~\cite{EKM13} proved that the \emph{swap distance} between two realizations $E$ and $F$ (the minimum number of $2$-switches required to transform $E$ into $F$) is at most $|E \triangle F|/2$.
More recently, Barrus and Haronian~\cite{BH23} investigated the structural conditions under which $\GF(d)$ contains a clique of size~$n$.
Regarding Hamiltonicity, Arikati and Peled~\cite{AP99} proved that $\GF(d)$ contains a Hamilton cycle whenever $d$ has majorization gap~$1$ (see~\cite{AP99} for the definition), and Barrus~\cite{BA16} showed that if $\GF(d)$ is triangle-free, then it is Hamiltonian.

\begin{figure}[ht]
    \centering
      \includegraphics[width=0.5\textwidth]{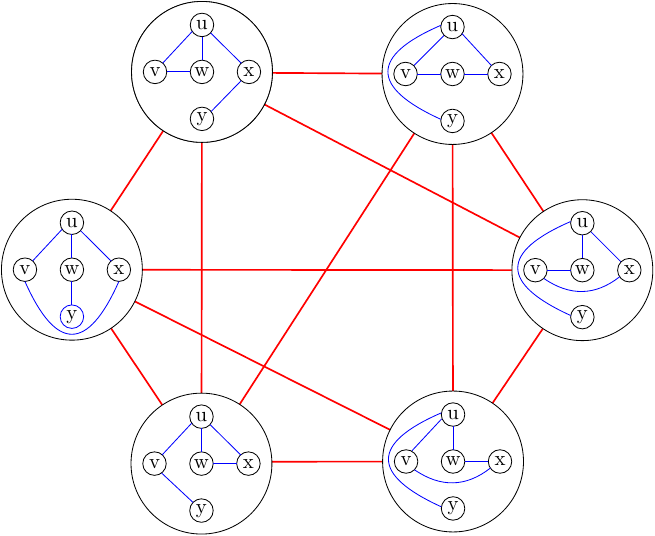}
	\caption{A realization graph $\GF((3,2,2,2,1))$. The blue lines show edges in each realization, whereas red lines show edges in the realization graph.}
    \label{figure:gf}
\end{figure}

Closely related to the realization graph is the \emph{interchange graph} $\mathcal{A}_{\mathcal F}(R,C)$ introduced by Brualdi~\cite{BR80}; see Figure~\ref{figure:af}.
For vectors $R = (r_1,\ldots,r_m)$ and $C = (c_1, \ldots, c_n)$ of non-negative integers, let $\mathcal{A}(R,C)$ be the set of all $m \times n$ $(0,1)$-matrices with row sums $R$ and column sums $C$, respectively.
The graph $\mathcal{A}_{\mathcal F}(R,C)$ has $\mathcal{A}(R,C)$ as its vertex set, and two matrices are adjacent if one can be transformed into the other by a single $2 \times 2$ interchange of the form
\[
\begin{pmatrix} 1 & 0 \\ 0 & 1 \end{pmatrix} \;\leftrightarrow\; \begin{pmatrix} 0 & 1 \\ 1 & 0 \end{pmatrix}.
\]
Ryser~\cite{Ryser57} and Gale~\cite{Gale57} showed that $\mathcal{A}_{\mathcal F}(R,C)$ is connected, and Chen, Guo, and Zhang~\cite{CGZ88} proved that its edge-connectivity equals its minimum degree.
As for the realization graph, the swap distance between two $(0,1)$-matrices $A$ and $B$ is at most half the number of entries in which they differ; for antipodal matrices (i.e., with $A + B$ the all-ones matrix), Yuster~\cite{YU05} proved the upper bound $\frac{1+\sqrt 2}{4+\sqrt 8}\, mn \,(1+o(1)) \approx 0.354\, mn\,(1+o(1))$ on this distance.
The diameter of $\mathcal{A}_{\mathcal F}(R,C)$ was studied in a series of papers: Brualdi~\cite{BR80} gave the first bound $\tfrac{mn}{2}-1$, Qian~\cite{QI99,QI02} improved it to $\tfrac{mn}{2} - \tfrac m2 \ln\bigl(\tfrac{n+2}{16}\bigr)$, and Brualdi and Shen~\cite{BS02} lowered it further to $\tfrac{3+\sqrt{17}}{16}\, mn \approx 0.445\,mn$.
In~1980, Brualdi~\cite{BR80} raised the question of whether the interchange graph admits a Hamilton path or cycle; for the special case in which $R$ or $C$ is the all-ones vector, Li and Zhang~\cite{LZ94} proved that $\mathcal{A}_{\mathcal F}(R,C)$ contains a Hamilton path.

\begin{figure}[ht]
    \centering
      \includegraphics[width=0.5\textwidth]{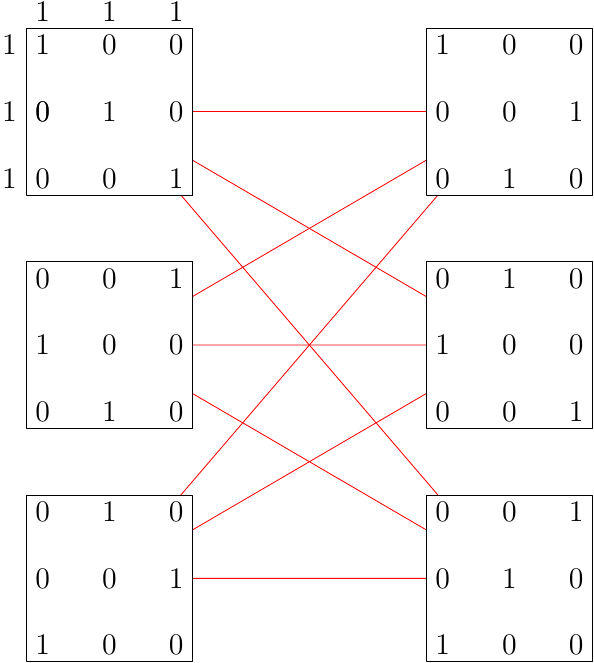}
	  \caption{An interchange graph $\mathcal{A_F}((1,1,1),(1,1,1))$.}
    \label{figure:af}
\end{figure}

The connection between the two families of graphs is an observation of Arikati and Peled~\cite{AP99}: 
For any vectors $R$ and $C$ of non-negative integers, there exists a degree sequence $d$ such that $\GF(d)$ and $\mathcal{A}_{\mathcal F}(R,C)$ are isomorphic.
Indeed, interpret an $m \times n$ matrix in $\mathcal{A}(R,C)$ as the biadjacency matrix of a bipartite graph with partition classes of sizes $m$ and $n$, and make the part of size $n$ a clique; see Figure~\ref{figure:embedding} for an example.
The resulting split graph has a degree sequence $d$ whose realization graph is isomorphic to $\mathcal{A}_{\mathcal F}(R,C)$: in a split graph in which the clique is one partite set, every $2$-switch removes and adds two edges that go between the independent set and the clique, which is exactly a $2\times 2$ interchange of the corresponding matrices.
Therefore, Theorem~\ref{thm_final} also answers Brualdi's question.

\begin{figure}[ht]
    \centering
      \includegraphics[width=0.9\textwidth]{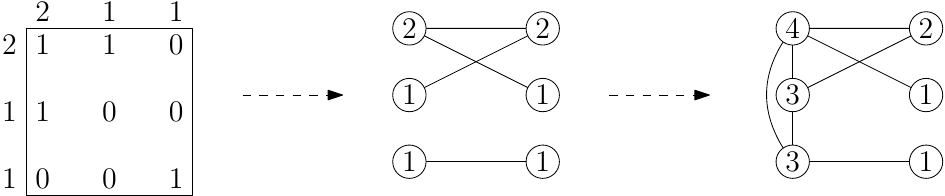}
	  \caption{An illustration of the conversion of a matrix in $\mathcal{A}((2,1,1),(2,1,1))$ to a realization in $\cG((4,3,3,2,1,1))$.}
    \label{figure:embedding}
\end{figure}

\begin{theorem}\label{ber_final}
For any non-negative integral vectors $R$ and $C$, the interchange graph $\mathcal{A}_{\mathcal F}(R,C)$ contains a Hamilton path starting from any matrix $A \in \mathcal{A}(R,C)$.
\end{theorem}

We remark that, independently and concurrently with our work, Baggett and Yan~\cite{BY26} also studied Brualdi's question.

Flip graphs of this kind have been studied extensively for other combinatorial objects, e.g., binary strings, permutations, subsets, matchings, spanning trees, and further structures; see the surveys of Savage~\cite{Sa97} and M\"utze~\cite{mutze2022combinatorial}, the book of Knuth~\cite{MR3444818}, and Nishimura's introduction to reconfiguration~\cite{ni18}.
Beyond pure combinatorics, $\GF(d)$ is the state space of the \emph{switch Markov chain}, the standard tool for sampling graphs with a prescribed degree sequence uniformly at random, e.g., in the study of large social networks.
The efficiency of this sampling method is governed by the \emph{mixing time} of the chain: Greenhill~\cite{Gr15} proved rapid mixing for degree sequences with minimum degree~$1$ and maximum degree $d_{\max}$ satisfying $3 \le d_{\max} \le \tfrac14 \sqrt M$, where $M$ is the sum of the degrees, and Amanatidis and Kleer~\cite{AK20} extended this to broader families, including strongly stable sequences.
A Hamilton path in $\GF(d)$ can be read as an explicit enumeration of all realizations by $2$-switches that visits each realization exactly once, so Theorem~\ref{thm_final} gives the strongest possible connectivity and traversal guarantee for the switch chain over all degree sequences.

\subsection{Overview of the proof}
The proof constructs a Hamilton path in $\GF(d)$ by induction on $n$; Figure~\ref{figure:induction_tree} illustrates the induction for the degree sequence $(3,2,2,2,1)$.
Given a prescribed starting realization $S$, we first choose a \emph{pivot} vertex $\mu$ (Lemma~\ref{lem:u}) and partition $\cG(d)$ according to the neighborhood of $\mu$.
Formally, let $\cN(d,\mu)$ denote the family of all subsets $A \subseteq [n] \setminus \{\mu\}$ that occur as the neighborhood of $\mu$ in some realization of $d$, and let $\cH(d,\mu,A)$ be the set of realizations in which the neighborhood of $\mu$ is exactly $A$.
Deleting $\mu$ from the realizations in $\cH(d,\mu,A)$ yields the degree sequence $d'$ obtained from $d$ by removing $d_\mu$ and decreasing the degrees of the vertices in $A$ by one; the corresponding induced subgraph $\HF(d,\mu,A)$ of $\GF(d)$ is isomorphic to $\GF(d')$.
By induction we obtain a Hamilton path in $\HF(d,\mu,A)$ with \emph{any prescribed endpoints}, for every $A \in \cN(d,\mu)$.

It remains to interconnect the paths of the parts $\HF(d,\mu,A)$.
Realizations in $\cH(d,\mu,A)$ and $\cH(d,\mu,B)$ can be neighbors only if $|A \triangle B| = 2$, i.e., $A$ and $B$ differ by swapping two elements.
We therefore study the \emph{neighborhood flip graph} $\NF(d,\mu)$ on $\cN(d,\mu)$, in which $A$ and $B$ are adjacent when $|A \triangle B| = 2$.
The \emph{subset flip graph} $\SF$ of a shifted family $\cS$: a family $\cS \subseteq \binom{N}{k}$ is called \emph{$N$-shifted} if it contains every $A \in \binom{N}{k}$ with $A \le B$ whenever it contains $B$, where $\le$ denotes the componentwise order on the increasingly sorted elements.

The key structural fact, proved in Section~\ref{sec:subset}, is that the subset flip graph of a shifted family contains a Hamilton path between any two of its sets, with \emph{exactly one exception}: the \emph{hook}, a shifted family with two maximal sets called its \emph{arms}, in which removing the two arms disconnects the remaining sets into two cliques.
In a hook, no Hamilton path can join one arm to the other, whereas Hamilton paths between all other pairs of sets do exist (Theorem~\ref{thm:shifted-hamilton}).
Classical Havel--Hakimi-type arguments show that $\cN(d,\mu)$ is a $([n] \setminus \{\mu\})$-shifted family for every degree sequence $d$ and every vertex $\mu$ (Lemma~\ref{left_leaning}), so $\NF(d,\mu)$ contains a Hamilton path between any two neighborhoods, unless it forms a hook with the two neighborhoods as its arms (Theorem~\ref{thm:neighborhood-hamilton}).

The pivot choice is what neutralizes the exceptional configuration: Lemma~\ref{lem:u} shows that $\mu$ can always be chosen, and the two neighborhoods that need to be joined are always chosen so that they are not both arms of a hook.
Consequently, the Hamilton paths of the parts $\HF(d,\mu,A)$ can be threaded along a Hamilton path of $\NF(d,\mu)$ into a single Hamilton path of $\GF(d)$ starting from $S$ (Lemma~\ref{lem:final}), which proves Theorem~\ref{thm_final}.

\begin{figure}[ht]
    \centering
    \includegraphics[width=0.8\textwidth]{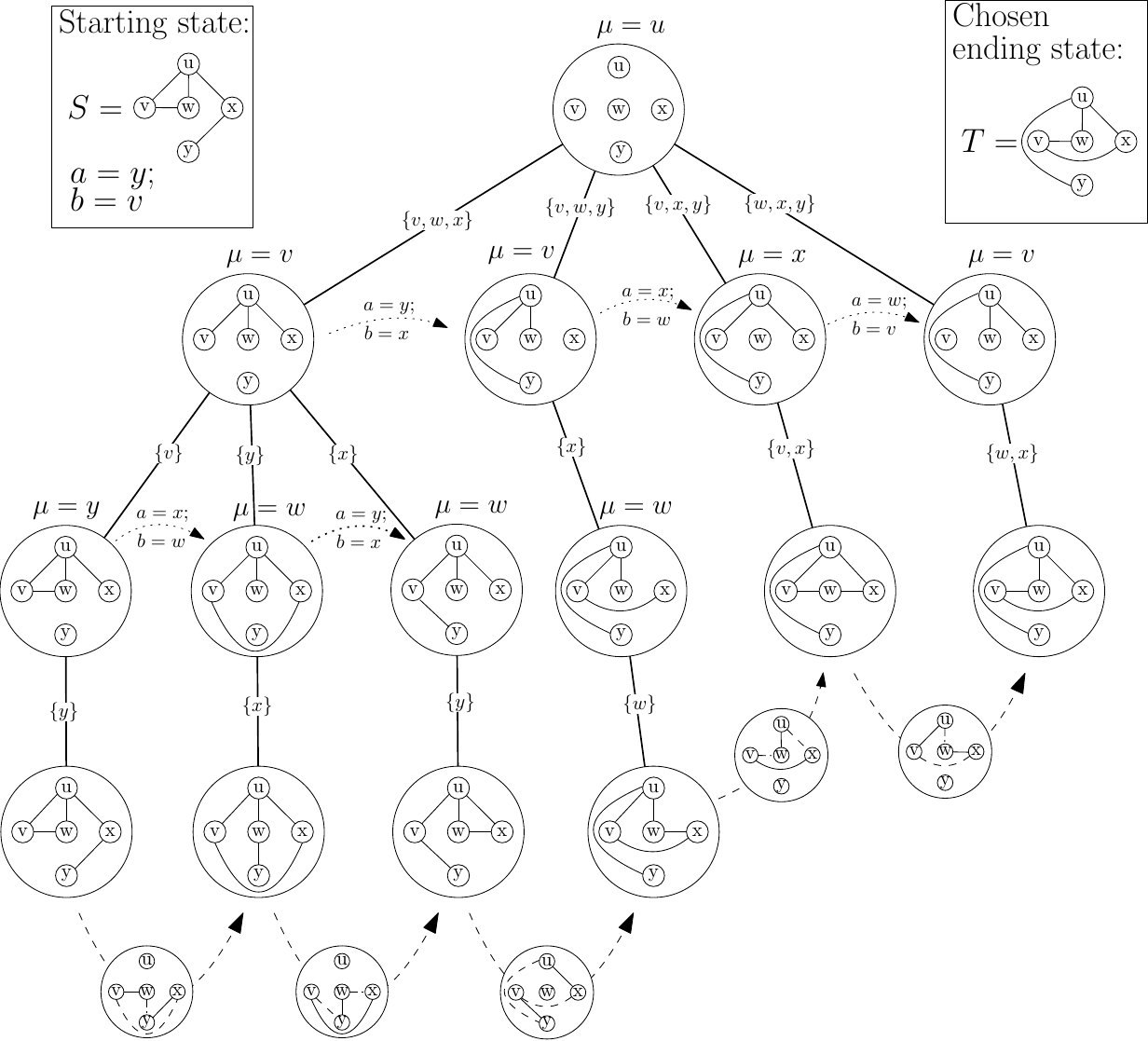}
	\caption{Induction tree for the construction of a Hamilton path in $\GF(d)$. Internal nodes correspond to choices of the pivot $\mu$; leaves correspond to individual realizations. Dotted arrows show transitions between neighborhoods of $\mu$; dashed arrows indicate connecting $2$-switches.}
    \label{figure:induction_tree}
\end{figure}

\section{Hamilton path in subset flip graphs} \label{sec:subset}

In this section, we characterize when the subset flip graph $\SF$ contains a Hamilton path between every two sets of an $N$-shifted family $\cS$.

For integers $1 \le k \le n$ and a set $N \subseteq [n]$, let $\binom{N}{k}$ denote the family of all $k$-element subsets of $N$.
For a set $A \subseteq [n]$, we denote by $A_i$ the $i$-th smallest integer present in $A$.
For any two sets $A,B \in \binom{[n]}{k}$, we write $A \le B$ if $A_i \le B_i$ for every $i \in [k]$.

For every family $\cS \subseteq \binom{N}{k}$, where $N$ is a subset of positive integers of size $n$, we can construct a family $\cS' \subseteq \binom{[n]}{k}$ such that $\cS'$ consists of the sets $\set{i \in [n];\; N_i \in A}$ for every $A \in \cS$.
In other words, we replace the ground set $N$ of $\cS$ with $[n]$, mapping elements of $N$ to $[n]$ while preserving the order of elements.
Clearly, the family $\cS$ is $N$-shifted if and only if $\cS'$ is $[n]$-shifted, and the subset flip graphs $\SF$ and $\SF'$ of the families $\cS$ and $\cS'$ are isomorphic.
Using this mapping, we can simplify the notation of the following proofs.

Note that every shifted family is uniquely determined by its maximal sets with respect to the partial order $\le$.
A hook $F(n, k; r, s) \subseteq \binom{[n]}{k}$ is the $[n]$-shifted family with maximal sets $[k-1] \cup \set{r}$ and $[k+1] \setminus \set{s}$, where $1 \le s \le k - 1$ and $k + 2 \le r \le n$.
Specifically, $F(n, k; r, s)$ consists of two sets $[k]$ and $[k-1] \cup \set{k+1}$, called the arms, together with the families $P = \set{[k-1] \cup \set{i};\; k+2 \le i \le r}$ and $Q = \set{[k+1] \setminus \set{i};\; s \le i \le k-1}$.

\begin{lemma}\label{lem:hook}
An $[n]$-shifted family $\cS \subseteq \binom{[n]}{k}$ forms a hook $F(n, k; r, s)$ for some $1 \le s \le k-1$ and $k + 2 \le r \le n$ if and only if $\cS$ contains sets $[k-1] \cup \set{k+2}$ and $[k+1] \setminus \set{k-1}$ and does not contain $[k-2] \cup \set{k,k+2}$.
\end{lemma}
\begin{proof}
First, assume that $\cS$ forms a hook $F(n, k; r, s)$.
Since $\cS$ is $[n]$-shifted and contains $[k-1] \cup \set{r}$ and $[k+1] \setminus \set{s}$, it contains sets $[k-1] \cup \set{k+2}$ and $[k+1] \setminus \set{k-1}$.
Since $[k-2] \cup \set{k,k+2} \not\le [k-1] \cup \set{r}$ and $[k-2] \cup \set{k,k+2} \not\le [k+1] \setminus \set{s}$, the set $[k-2] \cup \set{k,k+2}$ is not contained in $\cS$.

Second, assume that $\cS$ contains sets $[k-1] \cup \set{k+2}$ and $[k+1] \setminus \set{k-1}$ and does not contain $[k-2] \cup \set{k,k+2}$.
This implies that $\cS$ contains sets $[k-1] \cup \set{r}$ and $[k+1] \setminus \set{s}$ for some $1 \le s \le k-1$ and $k + 2 \le r \le n$.
For the sake of contradiction, assume that $\cS$ contains a set $A$ such that $A \not\le [k-1] \cup \set{n}$ and $A \not\le [k+1] \setminus \set{1}$.
Observe that $A \not\le [k-1] \cup \set{n}$ if and only if $A_{k-1} \ge k$, and $A \not\le [k+1] \setminus \set{1}$ if and only if $A_k \ge k+2$.
This implies that $[k-2] \cup \set{k,k+2} \le A$, so $[k-2] \cup \set{k,k+2} \in \cS$, which is a contradiction.
\end{proof}

The subset flip graph of $F(n, k; r, s)$ induces cliques on $P$ and on $Q$, with no edge between them; its arms are adjacent to each other and to every other set.
Since removing both arms from $F(n, k; r, s)$ creates two disconnected cliques $P$ and $Q$, the subset flip graph of $F(n, k; r, s)$ has no Hamilton path between the arms.
Observe that the subset flip graph has a Hamilton path between all other pairs of vertices.

The reversed complement of a set $A \subseteq [n]$ is the set $\bar A = \set{n-i+1;\; i \in [n] \setminus A}$, and the reversed complement of a family $\cS$ is $\bar \cS = \set{\bar A;\; A \in \cS}$.
Observe that for two sets $A, B \in \binom{[n]}{k}$ it holds that $A \le B$ if and only if $\bar A \le \bar B$, which implies that a family $\cS$ is $[n]$-shifted if and only if $\bar \cS$ is $[n]$-shifted.
Furthermore, the subset flip graph of a family is isomorphic to that of its reversed complement.
The reversed complement of $F(n, k; r, s)$ is $F(n, n-k; n-s+1, n-r+1)$.

\begin{theorem}\label{thm:shifted-hamilton}
Let $\cS \subseteq \binom{[n]}{k}$ be a shifted family and $A, B$ be two different sets of $\cS$.
Then, $\SF$ contains a Hamilton path between $A$ and $B$ unless $\cS$ forms a hook $F(n, k; r, s)$ for some $1 \le s \le k-1$ and $k + 2 \le r \le n$ and $A$ and $B$ are its arms.
\end{theorem}
\begin{proof}
If $k = 1$ or $k = n-1$, then $\SF$ is a complete graph, so the statement holds.
Therefore, we assume $2 \le k \le n-2$.
We proceed by induction on $n$, with base cases $n \le 3$.

Observe that if $\cS$ forms a hook, then the statement is satisfied.
Thus, we assume that $\cS$ is not a hook.
If $\cS$ does not contain $[k-1] \cup \set{k+2}$, then $\cS$ may contain only sets $[k+1] \setminus \set{s}$ for $s \in [k+1]$, so $\SF$ is a complete graph and the statement holds.
Therefore, we assume that $\cS$ contains $[k-1] \cup \set{k+2}$, and similarly, $\cS$ contains $[k+1] \setminus \set{k-1}$.
Since $\cS$ is not a hook, $\cS$ also contains $[k-2] \cup \set{k, k+2}$.

If all elements in $A \triangle B$ are larger than $k$, then all elements in $\bar A \triangle \bar B$ are smaller than or equal to $n-k$.
In this case, we apply the following proof to the reversed complement $\bar \cS \subseteq \binom{[n]}{n-k}$ to find a Hamilton path between $\bar A$ and $\bar B$ in $\bar \SF$.
Therefore, we assume that there exists an integer $i \le k$ in $A \triangle B$.
Without loss of generality, we assume that $i \in A \setminus B$.

We split the family $\cS$ into two families $\cS^+ = \set{C \in \cS;\; i \in C}$ and $\cS^- = \set{C \in \cS;\; i \notin C}$.
Let $\phi_- : \binom{[n] \setminus \set{i}}{k} \to \binom{[n-1]}{k}$ be a function that maps a set $A \in \binom{[n] \setminus \set{i}}{k}$ to $\set{j;\; j \in A,\, j < i} \cup \set{j-1;\; j \in A,\, j > i}$.
The inverse function $\phi_-^{-1}$ of $\phi_-$ preserves all elements of a set $A \in \binom{[n-1]}{k}$ that are smaller than $i$ and increases the other elements by one.
Similarly, we define a bijection $\phi_+$ between all sets in $\binom{[n]}{k}$ containing $i$ and all sets in $\binom{[n-1]}{k-1}$ by $\phi_+(A) = \set{j;\; j \in A,\, j < i} \cup \set{j-1;\; j \in A,\, j > i}$.
Although $\phi_-$ and $\phi_+$ are defined by the same formula, $\phi_+^{-1}$ differs from $\phi_-^{-1}$ by mapping sets of size $k-1$ instead of $k$ and by adding the integer $i$.
Observe that both $\phi_-(\cS^-)$ and $\phi_+(\cS^+)$, defined as $\set{\phi_-(A);\; A \in \cS^-}$ and $\set{\phi_+(A);\; A \in \cS^+}$, respectively, are shifted families, which allows us to apply induction.
Furthermore, the subset flip graphs of families $\cS^-$ and $\phi_-(\cS^-)$ are isomorphic, similarly for $\cS^+$ and $\phi_+(\cS^+)$.
We denote the subset flip graph of $\cS^-$ and $\cS^+$ by $\SF^-$ and $\SF^+$, respectively.

Our goal is to find sets $A' \in \cS^+$ and $B' \in \cS^-$ such that $|A' \triangle B'| = 2$ and $\SF^+$ contains a Hamilton path between $A$ and $A'$ and $\SF^-$ contains a Hamilton path between $B$ and $B'$.
Once we find such sets $A'$ and $B'$, we can join these Hamilton paths by the edge $A'B'$ to obtain a Hamilton path in $\SF$.
In order to apply the induction hypothesis to $\SF^+$ and $\SF^-$, we have to ensure that $A' \neq A$ and $B' \neq B$ (unless $|\cS^-| = 1$), and if $\cS^+$ is a hook, then not both $A, A'$ are its arms, and similarly if $\cS^-$ is a hook, then not both $B,B'$ are its arms.
Interested readers may observe that $|\cS^+| \ge 2$.

\medskip
\textbf{Case 1: $\phi_+(\cS^+)$ is a hook and $\phi_+(A)$ is its arm.}
Note that $k \ge 3$, since $\phi_+$ reduces the size of each set by one.
Since $\cS$ contains the set $[k-2] \cup \set{k, k+2}$, the set $\phi_+(\cS^+)$ can form a hook only if $i \ge k-1$.

Since $\phi_+(\cS^+)$ is a hook, it contains non-arm sets $[k] \setminus \set{k-2} = \phi_+([k+1] \setminus \set{k-2})$ and $[k-2] \cup \set{k+1} = \phi_+([k-2] \cup \set{i,k+2})$.
If $\phi_-(\cS^-)$ is a hook and $\phi_-(B)$ is its arm, then $\phi_-(\cS^-)$ contains a non-arm set $[k-2] \cup \set{k,k+1} = \phi_-([k-2] \cup \set{k+1,k+2})$ and we set $B' = [k-2] \cup \set{k+1,k+2}$ and $A' = [k-2] \cup \set{i,k+2}$.
Thus, we assume that $\phi_-(\cS^-)$ is not a hook or $\phi_-(B)$ is not its arm.
If $B \neq [k+1] \setminus \set{i}$, then we set $B' = [k+1] \setminus \set{i}$ and $A' = [k+1] \setminus \set{k-2}$; otherwise, we set $B' = [k+2] \setminus \set{i,k+1}$ and $A' = [k-2] \cup \set{i,k+2}$.

In the remaining cases, we assume that $\phi_+(\cS^+)$ is not a hook or $\phi_+(A)$ is not its arm.

\medskip
\textbf{Case 2: $\phi_-(\cS^-)$ is a hook and $\phi_-(B)$ is its arm.}
Since $\phi_-(\cS^-)$ contains a non-arm set $[k+1] \setminus \set{k-1} = \phi_-([k+2] \setminus \set{\min(i,k-1),k})$, we set $B' = [k+2] \setminus \set{\min(i,k-1),k}$.
We choose $j \in \set{k+1,k+2} \setminus \set{i}$ so that the set $A' = B' \triangle \set{j,i}$ is different from $A$.

In the remaining cases, we assume that $\phi_-(\cS^-)$ is not a hook or $\phi_-(B)$ is not its arm.

\medskip
\textbf{Case 3: $B \neq [k+1] \setminus \set{i}$.}
We set $B' = [k+1] \setminus \set{i}$, and choose $j \in \set{k-1,k,k+1} \setminus \set{i}$ so that $A' = [k+1] \setminus \set{j} = B' \triangle \set{j,i}$ is different from $A$.

\medskip
\textbf{Case 4: $B = [k+1] \setminus \set{i}$ and $|\cS^-| = 1$.}
In this case, the Hamilton path in $\SF^-$ contains only the set $B' = B$, and we have to find $A' \in \cS^+$ such that $|A' \triangle B'| = 2$ and $A' \neq A$.
Since $\cS$ contains sets $[k+1] \setminus \set{k-1}$, $[k+2] \setminus \set{k-1,k+1}$, $[k+1] \setminus \set{k}$, $[k+2] \setminus \set{k,k+1}$, it follows that $i \le k-2$, which implies that $\cS^+$ is not a hook.
We choose $j \in \set{k,k+1}$ so that $A' = [k+1] \setminus \set{j} = B' \triangle \set{j,i}$ is different from $A$.

\medskip
\textbf{Case 5: $B = [k+1] \setminus \set{i}$ and $|\cS^-| \ge 2$.}
We set $B' = [k+2] \setminus \set{i,k+1}$, which belongs to $\cS^-$, since $|\cS^-| \ge 2$ guarantees that some set in $\cS^-$ contains $k+2$.
We choose $j \in \set{k-1,k,k+2} \setminus \set{i}$ so that $A' = [k+2] \setminus \set{j,k+1} = B' \triangle \set{j,i}$ is different from $A$.
\end{proof}

\section{Isomorphism of neighborhood and subset flip graphs}

In this section, we analyze the structural properties of the neighborhood flip graph $\NF(d,\mu)$ for a given $\mu \in [n]$.
The following statement is based on classical results characterizing graphical sequences (see, e.g. \cite{H55,EG60}); it is included here for completeness.
Note that we always assume that degree sequences are in non-increasing order.
Let $k = d_\mu$ and $N = [n] \setminus \set{\mu}$.

\begin{lemma}\label{left_leaning}
For every degree sequence $d$ of length $n$ and vertex $\mu$, the family $\cN(d,\mu)$ is $([n] \setminus \set{\mu})$-shifted.
\end{lemma}
\begin{proof}
We prove that for every $A,B \in \binom{N}{k}$, if $A \le B$ and $B \in \cN(d,\mu)$, then $A \in \cN(d,\mu)$.
The proof proceeds by induction on the number of indices $i$ for which $A_i$ and $B_i$ differ.
If $A = B$, the statement trivially holds.

For the induction step, let $i \in [k]$ be the smallest index such that $A_i \neq B_i$.
In a realization $E \in \cH(d,\mu,B)$, the vertex $\mu$ is adjacent to vertex $B_i$ but not to vertex $A_i$, and $d_{A_i} \ge d_{B_i}$.
The pigeonhole principle guarantees the existence of a vertex $u \in [n] \setminus \set{\mu,A_i,B_i}$ such that $u$ is adjacent to $A_i$ but not to $B_i$ in $E$.
By performing a $2$-switch on vertices $\mu, A_i, u, B_i$, we obtain a new realization $F$ of $d$ in which the neighborhood of $\mu$ is $B' = (B \setminus \{B_i\}) \cup \{A_i\}$; see Figure~\ref{il:left_leaning}.
Since $A \leq B'$ and the number of differing indices between $A$ and $B'$ is strictly smaller than the number of differing indices between $A$ and $B$, we apply the induction hypothesis to prove $A \in \cN(d,\mu)$.
\end{proof}

\begin{figure}[ht]
    \centering
    \includegraphics[width=0.3\textwidth]{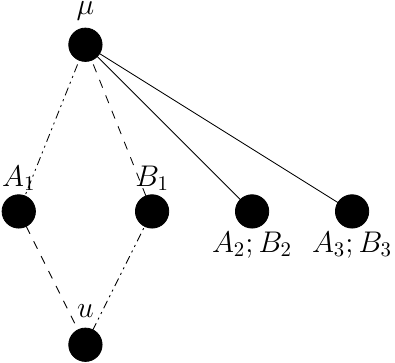}
	\caption{Illustration of Lemma~\ref{left_leaning}: a $2$-switch transforming $E \in \cH(d,\mu,B)$ into $F \in \cH(d,\mu,A)$ for $d=(3,1,1,1,1,1)$, $A = \{v,x,y\}$, and $B = \{w,x,y\}$. Dashed edges are unique to~$E$, dash-dotted to~$F$, and solid edges are common to both.}
    \label{il:left_leaning}
\end{figure}

The following theorem proves that the neighborhood flip graph $\NF(d,\mu)$ contains a Hamilton path between every two neighborhoods $A, B \in \cN(d,\mu)$ unless $\cN(d,\mu)$ forms a hook with arms $A$ and $B$.
However, the ground set of $\NF(d,\mu)$ is $[n] \setminus \set{\mu}$, so we need to define the hook on this ground set.
Formally, we define a hook $F([n] \setminus \set{\mu}, k; r, s)$ as the $([n] \setminus \set{\mu})$-shifted family with maximal sets $([k-1] \cup \set{r}) \setminus \set{\mu}$ and $[k+1] \setminus \set{s,\mu}$, where $1 \le s \le k - 1$ and $k + 2 \le r \le n$.
This means that the hook $F([n] \setminus \set{\mu}, k; r, s)$ is obtained from $F(n-1, k; r', s')$ by increasing all elements larger than $\mu$ by $1$, where $r' = r$ if $r < \mu$ and $r' = r - 1$ if $r > \mu$, and $s'$ is defined similarly.

\begin{theorem} \label{thm:neighborhood-hamilton}
For every degree sequence $d$ and vertex $\mu$,	the neighborhood flip graph $\NF(d,\mu)$ contains a Hamilton path between every two sets $A, B \in \cN(d,\mu)$ unless $\cN(d,\mu)$ forms a hook $F([n] \setminus \set{\mu}, k; r, s)$ with arms $A$ and $B$ for some integers $r,s$.
\end{theorem}
\begin{proof}
	The proof of this theorem follows directly from Lemma~\ref{left_leaning} and Theorem~\ref{thm:shifted-hamilton}.
\end{proof}

\section{Choosing a pivot vertex}

A tour on vertices of $[n]$ is a sequence of vertices $x_0, x_1, \ldots, x_m$ of $[n]$ such that $x_0 = x_m$ and $\set{x_i, x_{i+1}} \neq \set{x_j, x_{j+1}}$ for every $0 \le i < j \le m-1$; i.e., a tour may repeat vertices but not edges of the complete graph on vertex set $[n]$.
Note that we consider only closed tours.
A tour $x_0, \ldots, x_m$ on vertices of $[n]$ is alternating in a realization $E$ if $m$ is even and pairs $x_0 x_1, \ldots, x_{m-1} x_m$ are alternately edges and not edges of $E$.
Furthermore, for two realizations $E$ and $F$, a tour $x_0, \ldots, x_m$ on vertices of $[n]$ is alternating between $E$ and $F$ if $m$ is even and pairs $x_0 x_1, \ldots, x_{m-1} x_m$ are alternately edges of $E \setminus F$ and $F \setminus E$.

Let $C$ be the edges of an $E$-alternating tour $x_0, \ldots, x_m = x_0$.
By switching $C$ in $E$, we obtain another realization $E \triangle C$ with the same degree sequence as $E$.
Note that a $2$-switch is a special case of a switch along an alternating tour of length $4$.

Observe that for every two realizations $E$ and $F$ of the same degree sequence $d$, the set of edges $E \triangle F$ can be decomposed into a set of edge-disjoint tours alternating between $E$ and $F$.
However, the set of edges $E \triangle F$ may not be decomposable into alternating cycles (i.e., tours without repeating vertices) since $E \triangle F$ may consist of two cycles of odd length sharing a common vertex.

\begin{lemma}\label{lem:tour}
Let $S,T$ be two realizations of the same degree sequence $d$, and $e \in S \triangle T$ be an edge.
Then, there is a tour $x_0, \ldots, x_m = x_0$ alternating between $S$ and $T$ containing $e$ such that
\begin{itemize}
    \item every vertex occurs at most twice on the sequence $x_1, \ldots, x_m$, and
    \item if a vertex $x_i = x_j$ occurs twice on the tour for some $1 \le i < j \le m$ then $j - i$ is odd, and subtours $x_{i+1}, \ldots, x_{j-1}$ and $x_{j+1},\ldots,x_m=x_0,\ldots,x_{i-1}$ share no common vertex.
\end{itemize}
\end{lemma}

Note that $x_0 = x_m$ is considered a single occurrence on the tour.
Also note that the condition that $j-i$ is odd means that the two edges $x_i x_{i+1}$ and $x_{j-1} x_j$ belong to the same one of the sets $S \setminus T$ and $T \setminus S$.

\begin{proof}
Without loss of generality, we assume that $e \in S \setminus T$.
Furthermore, we assume that there is no tour alternating between $S$ and $T$ that does not contain all edges of $S \triangle T$.
Indeed, if such a tour $C$ existed, one of the pairs $(S, S \triangle C)$ and $(S, T \triangle C)$ would satisfy the assumption of the lemma, and we could consider this pair instead of $S$ and $T$ in the following proof.

We set $x_0$ and $x_1$ to be the endpoints of $e$.
We construct the tour by adding vertices one by one.
Since every vertex is incident with equally many edges of $S \setminus T$ and of $T \setminus S$, the only vertex $x_i$ where we may not be able to extend a sequence $x_0, \ldots, x_l$ alternating between $S$ and $T$ is $x_l = x_0$, and only if $x_{l-1} x_l \in T \setminus S$.

Now, we prove that the tour satisfies all the properties by contradiction.
If there exists $1 \le i < j \le m$ such that $x_i = x_j$ and $j-i$ is even, then $x_i, \ldots, x_j$ is a tour alternating between $S$ and $T$ that does not contain all edges of $S \triangle T$, which is a contradiction.
If one vertex is visited three times, i.e., $x_i = x_j = x_l$ for $1 \le i < j < l \le m$, then $j-i$ or $l-i$ or $l-j$ is even, a contradiction.

For the last property, assume that there exists $1 \le i < p < j < q \le m$ such that $x_i = x_j \neq x_p = x_q$.
Since $j - i$ is odd, edges $x_i x_{i+1}$ and $x_j x_{j+1}$ belong to different sets of $S \setminus T$ and $T \setminus S$, and similarly for $x_{p-1} x_p$ and $x_{q-1} x_q$.
Thus, the concatenation of $x_i, \ldots, x_p$ and the reversal of $x_j, \ldots, x_q$ provides a tour alternating between $S$ and $T$ that does not contain all edges of $S \triangle T$, a contradiction.
\end{proof}

\begin{theorem}[Havel~\cite{H55}, Hakimi~\cite{H62}]\label{Havel-Hakimi}
For an integer $i \in [n]$, a sequence $(d_1, \ldots, d_n)$ with $d_1 \ge \cdots \ge d_n \ge 0$ has a realization if and only if the sequence $(d'_1-1, \ldots, d'_{d_i}-1, d'_{d_i+1}, \ldots, d'_{n-1})$ has a realization where $d'_1, \ldots, d'_{n-1}$ denotes the sequence \\ $d_1, \ldots, d_{i-1},d_{i+1}, \ldots, d_n$.
\end{theorem}

This theorem implies that if $d$ is a degree sequence, then for every vertex $u$ there exists a realization of $d$ in which vertex $u$ is adjacent to the $d_u$ vertices with the largest degrees (excluding $u$).
We say that such a realization is obtained by applying the \emph{Havel--Hakimi theorem} to $u$.
Note that the Havel--Hakimi theorem can be used inductively, i.e., after applying the theorem to vertex $u$, we can apply it to another vertex $v$.

Let $(d'_1, \ldots, d'_{n-1})$ be the degree sequence obtained from $d$ by removing the degree of vertex $\mu$.
Recall that $N = [n] \setminus \set{\mu}$ and $k = d_\mu$.
Our notation ensures that the vertex $N_i$ has degree $d'_i$ for every $i \in [n-1]$.
Let $N[i:j] = \set{N_i, \ldots, N_j}$ for $1 \le i \le j \le n-1$.

\begin{lemma}\label{lem:gap}
Assume that $\NF(d,\mu)$ forms a hook $F([n] \setminus \{\mu\},k;r,s)$ for some integers $r,s$.
Then, $d'_{k-1} > d'_k \ge d'_{k+1} > d'_{k+2}$.
\end{lemma}
\begin{proof}
Since the sequence $d$ is sorted, we have $d'_{k-1} \ge d'_k \ge d'_{k+1} \ge d'_{k+2}$.

For the sake of contradiction, assume that $d'_{k-1} = d'_k$.
Since $\NF(d,\mu)$ forms a hook, there exists a realization $E$ in which $\mu$ is adjacent to vertices $N[1:k-1] \cup \set{N_{k+2}}$.
Since $d'_{k-1} = d'_k$, the bijection on $[n]$ that only swaps $N_{k-1}$ and $N_k$ gives a permutation mapping every realization of $d$ to a realization of $d$.
Therefore, this bijection maps $E$ to a realization in which $\mu$ is adjacent to vertices $N[1:k-2] \cup \set{N_k, N_{k+2}}$, which contradicts the assumption that $\NF(d,\mu)$ forms a hook.

The inequality $d'_{k+1} > d'_{k+2}$ follows by the same argument, applied to a realization $E$ in which $\mu$ is adjacent to vertices $N[1:k+1] \setminus \set{N_{k-1}}$.
\end{proof}

\begin{lemma}\label{lem:smaller_degree}
Assume that $\NF(d,\mu)$ forms a hook $F([n] \setminus \{\mu\},k;r,s)$ for some integers $r,s$.
Then, $d'_{k+1} \ge k$.
\end{lemma}
\begin{proof}
We prove the statement by contradiction, so we assume that $d'_{k+1} \leq k - 1$.
We construct a realization $E \in \cG(d)$ by setting the neighborhood of $\mu$ to be $N[1:k-1] \cup \set{N_{k+2}}$ and then applying the Havel--Hakimi theorem to $N_{k+1}$.
Lemma \ref{lem:gap} implies that the vertex $N_{k+1}$ is adjacent only to vertices of $N[1:k-1]$ in $E$.
We also construct a realization $F \in \cG(d)$ by setting the neighborhood of $\mu$ to be $N[1:k+1] \setminus \set{N_{k-1}}$.

Using Lemma \ref{lem:tour}, we construct a tour $x_0 = \mu, x_1 = N_k, \ldots, x_m = \mu$ alternating between $E$ and $F$ starting with the edge $\mu N_k$.
Let $C$ be the edges of the tour.
Observe that $x_{m-1}$ is either the vertex $N_{k-1}$ or the vertex $N_{k+2}$.

If $\mu N_{k+1} \notin C$, then $C$ contains exactly one of the edges $\mu N_{k-1}$ and $\mu N_{k+2}$.
Thus, if $\mu N_{k-1} \in C$ and $\mu N_{k+2} \notin C$, then the neighborhood of $\mu$ in the realization $E \triangle C$ is $N[1:k-2] \cup \set{N_k,N_{k+2}}$, which contradicts the assumption that $\NF(d,\mu)$ forms a hook.
If $\mu N_{k-1} \notin C$ and $\mu N_{k+2} \in C$, then the realization $F \triangle C$ leads to an analogous contradiction.
Therefore, $C$ contains all edges $\mu N_{k-1}$, $\mu N_k$, $\mu N_{k+1}$, and $\mu N_{k+2}$.

Let $1 \le i \le m-1$ be an integer such that $x_i = \mu$.
Since $\mu N_k, \mu N_{k+1} \in F \setminus E$, Lemma \ref{lem:tour} implies that $N_{k+1} = x_{i-1}$.
Since $x_{i-1}x_{i-2} \in E \setminus F$, the construction of $E$ implies that $x_{i-2} \in N[1:k-1]$.
Since $\mu x_{i-2} \in E$, the tour $x_0 = \mu, x_1 = N_k, \ldots, x_{i-2}, \mu$ is alternating in $E$; switching this tour in $E$ yields a realization in $\cH(d,\mu,N[1:k+2] \setminus \set{x_{i-2},N_{k+1}})$, which contradicts the assumption that $\NF(d,\mu)$ forms a hook.
\end{proof}

\begin{lemma}\label{lem:same_degree}
Assume that $\NF(d,\mu)$ forms a hook $F([n] \setminus \{\mu\},k;r,s)$ for some integers $r,s$.
Then, $d'_{k-1} > d'_k = d'_{k+1} = k > d'_{k+2}$.
\end{lemma}

\begin{proof}
Recall that $\cN(\bar d,n-\mu+1)$ forms a hook, where $\bar d$ denotes the reversed complement of $d$; here the pivot vertex $\mu$ is mapped to the vertex $n-\mu+1$ of degree $n-k-1$.
Let $\bar d'$ be the degree sequence obtained from $\bar d$ by removing the degree of vertex $n-\mu+1$.
Note that vertices $k$ and $k+1$ in $d'$ are mapped to vertices $n-k$ and $n-k-1$ in $\bar d'$.
Lemma \ref{lem:smaller_degree} implies that $\bar d'_{n-k} \ge \bar d_{n-\mu+1}$.
Since $\bar d'_{n-k} = n-1-d'_k$ and $\bar d_{n-\mu+1} = n-1-k$, we obtain $d'_k \le k$.
Together with Lemma~\ref{lem:smaller_degree}, we conclude that $k \ge d'_k \ge d'_{k+1} \ge k$.
Finally, Lemma \ref{lem:gap} states that $d'_{k-1} > d'_k \ge k > d'_{k+2}$, which concludes the proof.
\end{proof}

For three distinct vertices $a,b,u$, we say that a realization $E$ is $(u;a,b)$-adjacent if $E$ contains the edge $au$ and avoids $bu$.
For a set $X \subseteq [n]$ and integers $a, b \in [n]$, the notation $X \triangle \{a,b\}$ is used exclusively when $a \in X$ and $b \notin X$, representing the replacement of $a$ by $b$ in $X$.
For $A,B \in \cN(d,\mu)$ with $B = A \triangle \{ b,a \}$, Lemma~\ref{ab-swap} shows that the set $\cH(d,\mu,A)$ has a $(u;a,b)$-adjacent realization.
Note that performing a $2$-switch on vertices $u, a, \mu, b$ in a $(u;a,b)$-adjacent realization provides a realization in $\cH(d,\mu,B)$.

\begin{lemma}\label{lem:u}
Let $d$ be a degree sequence, $S \in \cG(d)$ be a realization, and $a,b,u$ be three distinct vertices such that $\cG(d)$ has a $(u;a,b)$-adjacent realization.
Then, there exists a vertex $\mu$ satisfying at least one of the following conditions.
    \begin{itemize}
        \item \textbf{Fixed neighbor condition:} Every realization in $\cG(d)$ is $(\mu;a,b)$-adjacent.
		\item \textbf{Swappable-edge condition:} Realization $S$ is not $(\mu;a,b)$-adjacent and there exists a $(\mu;a,b)$-adjacent realization $E$ in $\cG(d)$.
    \end{itemize}
Furthermore, if the Swappable-edge condition holds and $\NF(d,\mu)$ is a hook, then at least one of the neighborhoods $\set{x;x\mu \in S}$ and $\set{x; x\mu \in E}$ is not an arm of the hook.
\end{lemma}

\begin{proof}
The proof proceeds in two steps.
First, we find a vertex $\mu$ satisfying the Fixed neighbor or the Swappable-edge condition.
Second, if the first step yields a vertex $\mu$ satisfying the Swappable-edge condition but $\NF(d,\mu)$ is a hook and the neighborhoods $\set{x;x\mu \in S}$ and $\set{x; x\mu \in E}$ are the hook's arms, we replace the realization $E$ or vertex $\mu$ to satisfy the whole statement.

\paragraph{Step 1.}
We proceed by contradiction, so we assume that for every vertex $\mu$, neither the Fixed neighbor nor the Swappable-edge condition holds.
The assumption of the lemma provides a $(u;a,b)$-adjacent realization in $\cG(d)$ for some vertex $u$.
Since the Swappable-edge condition is not satisfied for $\mu = u$, $S$ is $(u;a,b)$-adjacent, so $u a \in S$ and $u b \notin S$.
Since the Fixed neighbor condition is not satisfied, there exists $E \in \cG(d)$ that is not $(u;a,b)$-adjacent.

\medskip
\textbf{Case 1: $au \notin E$.}
We apply Lemma \ref{lem:tour} to obtain a tour $C$ on vertices $x_0 = u, x_1 = a, \ldots, x_m = x_0$ alternating between $S$ and $E$ starting with the edge $ua$.
We consider two cases based on whether $x_2 = b$; see Figure~\ref{il:lemma_u}.

\textbf{Case 1.1: $x_2 \neq b$.}
Therefore, $av \in E \setminus S$ where $v = x_2$.
This implies that $S$ is not $(v;a,b)$-adjacent, so $\cG(d)$ does not contain a $(v;a,b)$-adjacent realization since we assume that the Swappable-edge condition is not satisfied for $\mu = v$.
Specifically, $E$ is not $(v;a,b)$-adjacent, so $vb \in E$.
Furthermore, $vb \notin S$; otherwise, performing a $2$-switch on vertices $u,a,v,b$ in $S$ would provide a $(v;a,b)$-adjacent realization.

From $vb \in E \setminus S$ it follows that $b \neq x_3$.
Observe that $bv \in C$; otherwise, switching the tour $x_0, \ldots, x_m$ in $S$ would provide a $(v;a,b)$-adjacent realization.
Therefore, there exists $4 \le i \le m-1$ such that $x_i = v$, and parity implies that $x_{i+1} = b$.
Lemma \ref{lem:tour} implies that sequences $x_{3},\ldots,x_{i-1}$ and $x_{i+1},\ldots,x_m,u,a$ share no common vertex, so $x_{i-1} x_{i+2} \notin C$.
Note that $v x_{i-1}, b x_{i+2} \in S \setminus E$.
The edge $x_{i-1} x_{i+2} \in S$; otherwise, switching the tour $u,a,v,x_{i-1},x_{i+2},b,u$ in $S$ would provide a $(v;a,b)$-adjacent realization.
However, switching the tour $x_0,\ldots,x_{i-1},x_{i+2},\ldots,x_m$ in $S$ provides a $(v;a,b)$-adjacent realization, which is a contradiction.

\textbf{Case 1.2: $x_2 = b$.}
Therefore, $ab \in E \setminus S$, and $bv \in S \setminus E$ where $v = x_3$.
Since $ub \notin S$, we have $u \neq v$.
Hence, $S$ is not $(v;a,b)$-adjacent.
Similarly to the previous case, we construct a $(v;a,b)$-adjacent realization in $\cG(d)$, which will contradict the Swappable-edge condition for $\mu = v$.
The edge $av \in S$; otherwise, performing a $2$-switch on $u,a,v,b$ in $S$ would provide a $(v;a,b)$-adjacent realization.

We prove that $av \notin C$ by contradiction.
Since $av \in S$, there exists $i \ge 5$ such that $x_i = v$.
Since $x_3 = x_i = v$, Lemma \ref{lem:tour} implies that $i$ is even and $x_{i+1} = a$.
However, $x_1 = x_{i+1} = a$, so Lemma \ref{lem:tour} implies that $i$ is odd, which is a contradiction.

Since $av \notin C$, $S \triangle C$ is a $(v;a,b)$-adjacent realization, which is a contradiction.

\medskip
\textbf{Case 2: $au \in E$.}
Since $E$ is not $(u;a,b)$-adjacent, $bu \in E$. 
To apply the argument of Case 1 to the reversed complement $\bar d$, we first verify that all assumptions are satisfied.
The reversal maps vertices $u, a, b$ onto vertices denoted by $\bar u, \bar a, \bar b$; i.e., $\bar u = n+1-u$.
Note that the vertex $\bar b$ takes the role of $a$ and $\bar a$ takes the role of $b$ in Case 1.
Indeed, the complement of $S$ is a $(\bar u; \bar b, \bar a)$-adjacent realization in $\cG(\bar d)$.
Since $\bar b \bar u \notin \bar E$, the realization $\bar E \in \cG(\bar d)$ is not $(\bar u; \bar b, \bar a)$-adjacent.
The proof of Case 1 provides a vertex $\bar \mu$ such that every realization in $\cG(\bar d)$ is $(\bar \mu; \bar b, \bar a)$-adjacent (the Fixed neighbor condition) or the realization $\bar S$ is not $(\bar \mu; \bar b, \bar a)$-adjacent and there exists a $(\bar \mu; \bar b, \bar a)$-adjacent realization in $\cG(\bar d)$ (the Swappable-edge condition).
Hence, the vertex $\mu = n+1-\bar\mu$ satisfies the Fixed neighbor or the Swappable-edge condition in $\cG(d)$.

\medskip
\paragraph{Step 2.}
Let $A = \setcond{x}{x\mu \in S}$ and $B = \setcond{x}{x\mu \in E}$.
Assume that Step 1 provides a vertex $\mu$ satisfying the Fixed neighbor or the Swappable-edge condition.
If the Fixed neighbor condition holds, or $\NF(d,\mu)$ does not form a hook, or $A,B$ are not its arms, the statement holds.
So we assume that the Swappable-edge condition holds and $\NF(d,\mu)$ forms a hook and $A,B$ are its arms.
Therefore, $A \triangle B = \set{N_k,N_{k+1}}$.
Since $E$ is $(\mu; a, b)$-adjacent while $S$ is not, it follows that $\set{a,b} \cap \set{N_k,N_{k+1}} \neq \emptyset$.

Since $E$ is $(\mu;a,b)$-adjacent, it holds that $a \in B$ and $b \notin B$, which implies that $a \le N_{k+1}$ and $b \ge N_k$.
Similarly, since $S$ is not $(\mu;a,b)$-adjacent, it holds that $a \notin A$ or $b \in A$, which implies that $a \ge N_k$ or $b \le N_{k+1}$.

\medskip
\textbf{Case 1: $\set{a,b} \neq \set{N_k, N_{k+1}}$.}
We define
$$
B' = 
\begin{cases}
N[1:k+1] \setminus \set{N_{k-1}} & \text{if $a \ge N_k$} \\
N[1:k-1] \cup \set{N_{k+2}} & \text{otherwise.}
\end{cases}
$$
Lemma \ref{lem:hook} implies that $B' \in \cN(d,\mu)$, so there exists a realization $E' \in \cH(d,\mu,B')$.
Since $a \in B'$ and $b \notin B'$, the realization $E'$ is $(\mu;a,b)$-adjacent, so the statement holds.

\medskip
\textbf{Case 2: $\set{a,b} = \set{N_k, N_{k+1}}$.}

Observe that $a \mu \notin S$ and $b \mu \in S$.
Note that the only neighborhoods in $\cN(d,\mu)$ containing $a$ and avoiding $b$ are the two arms; hence, unlike in Case 1, we cannot satisfy the statement by choosing another set $B' \in \cN(d,\mu)$.
Instead, we replace the vertex $\mu$ by another vertex $u$ such that $S$ is not $(u;a,b)$-adjacent and $\cG(d)$ contains a $(u;a,b)$-adjacent realization to satisfy the Swappable-edge condition.

Now, we prove the following five properties of the realization $S$.

\textbf{Property 1: No two vertices of $N[k+2:n]$ are connected by an edge in $S$.}
For the sake of contradiction, consider two vertices $x,y \in N[k+2:n]$ with $xy \in S$.
By Lemma \ref{lem:same_degree}, the degree of $x$ is at most $k-1$, so there exists a vertex $z \in N[1:k-1]$ with $xz \notin S$.
Performing a $2$-switch on vertices $x,y,\mu,z$ provides a realization whose neighborhood of $\mu$ is $N[1:k-1] \setminus \set{z} \cup \set{b,y}$, which contradicts the assumption that $\NF(d,\mu)$ forms a hook.

\textbf{Property 2: Every two vertices of $N[1:k-1]$ are connected by an edge in $S$.}
If $x,y \in N[1:k-1]$ are not neighbors in $S$, then $x$ is connected to a vertex $z \in N[k+2:n]$, and a $2$-switch on vertices $y,\mu,z,x$ leads to an analogous contradiction.

\textbf{Property 3: The vertex $a$ is connected to all vertices of $N[1:k-1]$.}
If there exists a vertex $x \in N[1:k-1]$ not connected to $a$, then there exists a vertex $y \in N[k+2:n]$ connected to $a$, and a $2$-switch on vertices $a,y,\mu,x$ leads to an analogous contradiction.

\textbf{Property 4: The vertex $b$ is not connected to any vertex of $N[k+2:n]$.}
If there exists a vertex $x \in N[k+2:n]$ connected to $b$, then there exists a vertex $y \in N[1:k-1]$ not connected to $b$, and a $2$-switch on vertices $x,b,y,\mu$ leads to an analogous contradiction.

\textbf{Property 5: There exists exactly one vertex $u \in [n] \setminus \set{a,b,\mu}$ such that $S$ is $(u;a,b)$-adjacent.}
The existence of $u$ follows from the pigeonhole principle: $a$ and $b$ have the same degree $k$, $a\mu \notin S$, and $b\mu \in S$.
Property 3 implies that there is at most one vertex in $N[k+2:n]$ connected to $a$, and Property 4 implies that there is at most one vertex in $N[1:k-1]$ not connected to $b$.
If the vertex $a$ is connected to some vertex in $N[k+2:n]$, then $ab \notin S$, and $b$ is connected to all vertices in $N[1:k-1]$.
Therefore, the vertex $u$ connected to $a$ but not connected to $b$ is unique.

Actually, we proved the following stronger property:
Exactly one of the following two conditions holds.
\begin{itemize}
    \item $u \ge N_{k+2}$, $ab \notin S$, and the neighborhood of $a$ and $b$ in $S$ is $N[1:k-1] \cup \set{u}$ and $N[1:k-1] \cup \set{\mu}$, respectively.
    \item $u \le N_{k-1}$, $ab \in S$, and the neighborhood of $a$ and $b$ in $S$ is $N[1:k-1] \cup \set{b}$ and $N[1:k-1] \setminus \set{u} \cup \set{\mu,a}$, respectively.
\end{itemize}
Observe that one condition holds for $d$ and other one holds for the reversed complement $\bar d$ (assuming we exchange roles of $a$ and $b$, replace $k$ by $n-k-1$, and $S$ by $\bar S$).
Therefore, without loss of generality, we assume that $u \le N_{k-1}$.

We construct a realization $E' \in \cG(d)$ by setting the neighborhood of $\mu$ to be $N[1:k-1] \cup \set{N_{k+2}}$ and then applying the Havel--Hakimi theorem to $N_k$ and then to $N_{k+1}$.
Lemma \ref{lem:same_degree} implies that the neighborhood of $a$ and $b$ in $E'$ is $N[1:k-1] \cup \set{b}$ and $N[1:k-1] \cup \set{a}$, respectively.
Let $v = N_{k+2}$.
Performing a $2$-switch on vertices $a,b,\mu,v$ in $E'$ provides a $(v;a,b)$-adjacent realization, denoted by $E$.
Therefore, that the realization $E$ is $(v;a,b)$-adjacent while $S$ is not, so the Swappable-edge condition is satisfied for the vertex $v$ and the realization $E$.

We prove that they satisfy the last requirement of the statement by a contradiction.
Thus, we assume that $\cN(d,v)$ forms a hook and the sets $A' = \setcond{x}{xv \in S}$ and $B' = \setcond{x}{xv \in E}$ are its arms.
Let $N' = [n] \setminus \set{v}$.
Similarly as before, $A' \triangle B' = \set{N'_{d_v},N'_{d_v+1}}$ and $\set{a,b} \cap \set{N'_{d_v},N'_{d_v+1}} \neq \emptyset$.
However, Lemma \ref{lem:same_degree} implies that $d_v = d_{N'_{d_v}} = d_{N'_{d_v+1}} \neq d_\mu = d_a = d_b$, which is a contradiction.
\end{proof}

\begin{figure}[ht]
    \centering
    \includegraphics[width=0.5\textwidth]{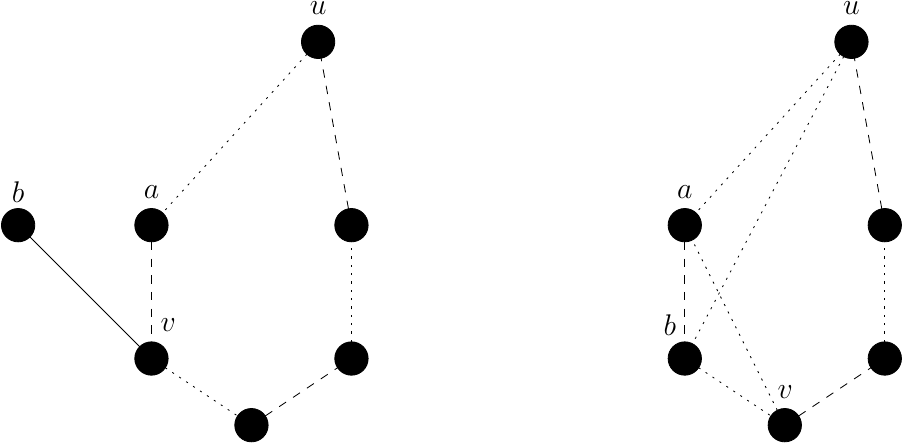}
	\caption{Illustration of Lemma~\ref{lem:u}. 
    Dashed edges belong only to $E$, dash-dotted edges only to $S$, solid edges to both $E$ and $S$, and dotted edges to neither.
    Left: case $ab \notin C$.
    Right: case $ab \in C$.
    In both cases, a $2$-switch on $\{\mu,a,v,b\}$ in $S$ yields a realization containing $av$ but avoiding $vb$.}
    \label{il:lemma_u}
\end{figure}

\section{Hamilton path in realization graphs}

To prove Lemma \ref{lem:final}, which implies Theorem \ref{thm_final}, we first establish the following lemma.

\begin{lemma}\label{ab-swap}
For every $A,B \in \cN(d,\mu)$ satisfying $\{a\} = B \setminus A$ and $\{b\} = A \setminus B$ for some vertices $a,b$, the set $\cH(d,\mu, A)$ contains a $(u;a,b)$-adjacent realization for some vertex $u$ distinct from $\mu$.
\end{lemma}

\begin{proof}
Let $E \in \cH(d,\mu,A)$ and $F \in \cH(d,\mu,B)$ be realizations.
By Lemma \ref{lem:tour}, let $x_0 = \mu, x_1 = a, \ldots, x_m$ be a tour alternating between $E$ and $F$ that starts with the edge $\mu a$.
Note that $E \triangle F$ contains exactly two edges incident with $\mu$, namely $a\mu \in F \setminus E$ and $b\mu \in E \setminus F$, so $x_{m-1} = b$.
Let $u = x_2$, which implies that $au \in E \setminus F$.
Clearly, $u$ is different from $a$ and $\mu$, but $u$ may be equal to $b$.

\textbf{Case 1: $u \neq b$.}
If $bu \notin E$, then, by definition, $E$ is $(u;a,b)$-adjacent.
If $bu \in E$, then $x_2 = u, \ldots, x_{m-1} = b, u$ is a tour alternating in $E$; switching this tour in $E$ yields a $(u; a, b)$-adjacent realization in $\cH(d,\mu,A)$.

\textbf{Case 2: $u = b$.}
Let $v = x_3$, which implies that $bv \in F \setminus E$.
Note that $v$ is different from $\mu, a, b$.
If $av \in E$, then, by definition, $E$ is $(v;a,b)$-adjacent.
If $av \notin E$, then $x_3 = v, \ldots, x_{m-1} = b, a, v$ is a tour alternating in $E$; switching this tour in $E$ yields a $(v; a, b)$-adjacent realization in $\cH(d,\mu,A)$.
\end{proof}

\begin{lemma}\label{lem:final}
Let $d$ be a degree sequence, $S$ be its realization, and $a,b,u$ be three distinct vertices such that $\cG(d)$ has a $(u;a,b)$-adjacent realization.
Then, $\GF(d)$ has a Hamilton path from $S$ to some $(v;a,b)$-adjacent realization for some vertex $v$.
\end{lemma}

\begin{proof}
We prove the statement by induction on the length $n$ of the sequence $d$.
If there is only one realization in $\cG(d)$, then the statement of the lemma holds; so we assume that there are at least two realizations.
Note that $\cG(d)$ has only one realization if $n \le 3$, which provides the base of the induction.

Let vertices $u,a,b$ and the realization $S$ be given as in the statement such that $\cG(d)$ has a $(u;a,b)$-adjacent realization.
By Lemma \ref{lem:u}, there exists a vertex $\mu$ satisfying the Fixed neighbor or the Swappable-edge condition.
Let $l = |\cN(d,\mu)|$ and let $A_1$ be the neighborhood of $\mu$ in $S$.

If the Swappable-edge condition holds, then Lemma \ref{lem:u} provides a $(\mu;a,b)$-adjacent realization $E$; we then let $A_l$ be the neighborhood of $\mu$ in $E$.
Observe that $A_1 \neq A_l$ since $A_1 \triangle A_l$ contains $a$ or $b$.
Furthermore, Lemma \ref{lem:u} ensures that if $\cN(d,\mu)$ forms a hook, then $A_1$ and $A_l$ are not both its arms.
If the Fixed neighbor condition holds, then choose $A_l \in \cN(d,\mu)$ such that $A_1 \neq A_l$ unless $l = 1$.
If $\cN(d,\mu)$ forms a hook, then we also choose $A_l$ not to be its arm.
For both conditions, observe that all realizations in $\cH(d,\mu,A_l)$ are $(\mu;a,b)$-adjacent.
We construct a Hamilton path in $\GF(d)$ whose endpoint lies in $\cH(d,\mu,A_l)$, so that the path terminates in a $(\mu;a,b)$-adjacent realization as required.

By Theorem~\ref{thm:neighborhood-hamilton}, the graph $\NF(d,\mu)$ contains a Hamilton path $A_1, A_2, \ldots, A_l$.
For every $i \in [l]$, we iteratively construct a Hamilton path in $\HF(d,\mu,A_i)$ from $S_i \in \cH(d,\mu,A_i)$ to $T_i \in \cH(d,\mu,A_i)$ such that $T_i$ and $S_{i+1}$ are neighbors in $\cG(d)$ and $S_1 = S$ as follows.

Fix $i \in [l-1]$ and suppose a realization $S_i \in \cH(d,\mu,A_i)$ is given.
Since $A_i$ and $A_{i+1}$ are neighbors in $\NF(d,\mu)$, there are vertices $a_i,b_i \in [n] \setminus \{\mu\}$ such that $\{b_i\} = A_i \setminus A_{i+1}$ and $\{a_i\} = A_{i+1} \setminus A_i$.
By Lemma~\ref{ab-swap}, $\cH(d,\mu,A_i)$ has a $(u_i;a_i,b_i)$-adjacent realization for some vertex $u_i \neq \mu$.
Recall that $\HF(d,\mu,A_i)$ is isomorphic to $\GF(d')$, where $d'$ is obtained from $d$ by removing $d_\mu$ and decrementing the degrees of all vertices in $A_i$ by one.
Thus, by the induction hypothesis applied to the degree sequence $d'$, there exists a Hamilton path in $\cH(d,\mu,A_i)$ from $S_i$ to some $(v_i;a_i,b_i)$-adjacent realization $T_i$ for some vertex $v_i$.
Let $S_{i+1} \in \cH(d,\mu,A_{i+1})$ be the realization obtained from $T_i$ by performing a $2$-switch on vertices $a_i, v_i, b_i, \mu$.

For $i = l$, we construct a Hamilton path in $\cH(d,\mu,A_l)$ from $S_l$ using the induction hypothesis.
The concatenation of these Hamilton paths provides the desired Hamilton path in $\GF(d)$.
\end{proof}

Lemma~\ref{lem:final} directly implies Theorem~\ref{thm_final}.

\begin{proof}[Proof of Theorem~\ref{thm_final}]
If $|\cG(d)| = 1$, then $\cG(d)$ has a trivial Hamilton path.
Note that $|\cG(d)| = 1$ includes the cases in which $\cG(d)$ contains the complete or empty graph.
Otherwise, $\cG(d)$ has a $(u;a,b)$-adjacent realization for some vertices $a,b,u$, and from Lemma~\ref{lem:final} we obtain a Hamilton path in $\GF(d)$ starting from a given realization.
\end{proof}

\bigskip
\textbf{Acknowledgment}
We would like to thank Jeff Baggett for bringing the error in the previous version of this paper to our attention and for suggesting the statement of Theorem \ref{thm:shifted-hamilton} as a correction.

A preliminary version of this paper was submitted to SVO\v{C} competition (\url{https://home.pf.jcu.cz/~svoc2026/program.html}) and was used in the bachelor's thesis of Petr Hlad\'ik.
During the final preparation of the journal version, we learned about an independent paper by Baggett and Yan~\cite{BY26}.

The authors acknowledge the use of large language models for assistance in improving the clarity, grammar, and overall readability of the English text.
All technical content, ideas, results, and conclusions presented in this work are the original contributions of the authors.

\bibliographystyle{plain}
\bibliography{references}

\end{document}